\documentclass[preprint,12pt,3p,sort&compress]{elsarticle}
\usepackage[ruled,vlined]{algorithm2e}
\include{pythonlisting}
\usepackage{csquotes}
\usepackage{caption}
\usepackage{subcaption}
\usepackage{easyReview}
\usepackage{natbib}
\usepackage{amssymb,bm}
\usepackage[mathcal]{euscript}
\usepackage{epstopdf}
\usepackage{amsthm,amsmath,mathtools}

\newtheorem{Theorem}{Theorem}
\usepackage{color}
\usepackage{mathrsfs}
\usepackage{dsfont}
\usepackage{graphicx}
\usepackage[shortlabels]{enumitem}
\usepackage{empheq} 
\usepackage{float}
\usepackage{physics}
\usepackage[linkcolor=blue, urlcolor=blue, citecolor=blue,
colorlinks, bookmarks]{hyperref}
\usepackage{amsfonts}
\usepackage{multirow,booktabs}
\usepackage{adjustbox}
\usepackage{cases} 
\allowdisplaybreaks
\numberwithin{equation}{section}
\numberwithin{Theorem}{section}

\numberwithin{Definition}{section}
\newtheorem{lemma}{Lemma}

\numberwithin{Lemma}{section} 
\numberwithin{Proposition}{section}

\numberwithin{Remark}{section}
\numberwithin{Example}{section}

\numberwithin{Corollary}{section}
\journal{.}
\usepackage{array}
\newcolumntype{L}{>{$}l<{$}} 
\newcolumntype{C}{>{$}c<{$}} 
\begin{document}	
	\begin{frontmatter}
		\title{Error analysis of a high-order fully discrete method for two-dimensional time-fractional convection-diffusion equations exhibiting weak initial singularity}
		\author[label1]{Anshima Singh}
		\address[label1]{Department of Mathematical Sciences, Indian Institute of Technology (BHU) Varanasi, Uttar Pradesh, India}
		\ead{anshima.singh.rs.mat18@itbhu.ac.in}
		\author[label1]{Sunil Kumar}
		\ead{skumar.iitd@gmail.com}
\begin{abstract}
 This study presents a novel high-order numerical method designed for solving the two-dimensional time-fractional convection-diffusion (TFCD) equation. The Caputo definition is employed to characterize the time-fractional derivative. A weak singularity at the initial time ($t=0$) is encountered in the considered problem, which is effectively managed by adopting a discretization approach for the time-fractional derivative, where Alikhanov's high-order L2-1$_\sigma$ formula is applied on a non-uniform fitted mesh, resulting in successful tackling of the singularity. A high-order two-dimensional compact operator is implemented to approximate the spatial variables. The alternating direction implicit (ADI) approach is then employed to solve the resulting system of equations by  decomposing the two-dimensional problem into two separate one-dimensional problems. The theoretical analysis, encompassing both stability and convergence aspects, has been conducted comprehensively, and it has shown that method is convergent with an order $\mathcal O\left(N_t^{-\min\{3-\alpha,\theta\alpha,1+2\alpha,2+\alpha\}}+h_x^4+h_y^4\right)$, where $\alpha\in(0,1)$ represents the order of the fractional derivative, $N_t$ is the temporal discretization parameter and $h_x$ and $h_y$ represent spatial mesh widths. Moreover, the parameter $\theta$ is utilized in the construction of the fitted mesh.
\end{abstract}
\end{frontmatter}
\section{Introduction}\label{sec:1}
The study of partial differential equations (PDEs) plays a pivotal role in understanding the fundamental principles governing various physical phenomena. Among the wide range of PDEs, the convection-diffusion equation holds significant importance due to its ability to model diverse phenomena, including fluid flow, heat transfer, and mass transport \cite{morton2019revival,brent1988enthalpy,xue2013interval,carroll2010experimental} etc. In recent years, the extension of traditional PDEs to fractional-order PDEs has attracted considerable attention due to its ability to capture anomalous diffusion and long-term memory effects, resulting in more accurate models for complex real-world processes.

The time-fractional convection-diffusion equation has found a multitude of real-world applications, revolutionizing research in various fields. In environmental science, it enables accurate modeling of pollutant dispersion in air and water systems, aiding in the development of effective pollution control strategies \cite{chen2020adaptive,hariharan2014efficient}. In biomedical research, it helps to understand drug diffusion in tissues and the spread of diseases within populations \cite{hose2004thermal,katz2015vaginal}. Moreover, it finds utility in materials engineering to optimize heat and mass transfer in complex media, such as porous materials and composite structures \cite{li2014conjugate,khan2018multiphase,zhang2019fractional}. By incorporating fractional derivatives, the equation provides a powerful tool for unraveling intricate dynamics and capturing anomalous diffusion behavior, facilitating a deeper understanding of real-world phenomena and opening doors to innovative solutions in diverse scientific domains.

Therefore, in this paper, we direct our focus towards the following two-dimensional time-fractional convection-diffusion equation
\begin{align}\label{1}
	\partial^{\alpha}_{t}u=\lambda_1 u_{xx}+\lambda_2 u_{yy}+\mu_1u_x+\mu_2u_y+\gamma u+f,~~  \mbox{in}~\Omega \times (0,T_f],~\Omega \subset \mathds{R}^2,
\end{align}
with initial and boundary conditions
\begin{equation}\label{2}
	\left\{%
	\begin{array}{ll}
		u(x,y,0)=\phi(x,y),&(x,y)\in \bar{\Omega}=\Omega\cup \partial\Omega,\\
		u(x,y,t)=\psi(x,y,t),&({x},y)\in \partial\Omega,~0<t\leq T_f,
	\end{array}%
	\right.
\end{equation}
where $\partial{\Omega}$ denotes the boundary of $\Omega=(0,L)\times(0,L),$ and
\begin{equation}\label{3}
	\partial^{\alpha}_{ t}u(x,y,t)=\frac{1}{\Gamma(1-\alpha)} \int_{0}^{t}(t-\mathcal C)^{-\alpha} \frac{\partial u(x,y,\mathcal C)}{\partial\mathcal C}\,d\mathcal C~
\end{equation}
defines the Caputo derivative of fractional order $\alpha\in (0,1)$. Also, $\lambda_1$ and $\lambda_2$ are positive constants, and $\mu_1$, $\mu_2,$ and $\gamma$ are constants such that $\frac{\mu_1^2}{4\lambda_1}+\frac{\mu_2^2}{4\lambda_2}-\gamma\geq0$. Further $f$, $\phi$, and $\psi$ are sufficiently smooth functions in their respective domains.

Unraveling the analytical solutions to PDEs entailing fractional order derivatives often poses an intricate challenge. Hence, the need arises to adapt numerical techniques as a means of approximating solutions to these complex equations. While research on numerical approximations for solutions of fractional partial differential equations (FPDEs) in one spatial dimension has gained substantial traction in recent years, exploration of FPDEs with higher spatial dimensions remains relatively scarce. This underscores the significance of developing adept numerical techniques tailored for higher dimensional problems, fueling the pursuit of innovative solutions in this realm.

It is widely recognized that a typical solution to equation (\ref{1})-(\ref{2}) exhibits a singularity in proximity to the initial time $t= 0$. Additionally, it's derivatives meet the following regularity conditions given by \cite{stynes2017error}:
\begin{align}
	\left|\frac{\partial^{l}u}{\partial z^l}\right|&\leq \kappa_1,~ \text{for}~l=0,1,2,3,4,5,6,\label{4}\\
	\left|\frac{\partial^{l}u}{\partial t^l}\right|&\leq \kappa_2(1+t^{\alpha-l}),~ \text{for}~l=0,1,2,3,\label{5}
\end{align}
where constants $\kappa_1$ and $\kappa_2$ are positive that does not depend on both time ($t$) and spatial ($z=x,y$) variables. The expression given by Equation (\ref{5}) suggests that $u(x,y,t)$ exhibits a weak singularity at $t= 0$, resulting in the time derivative $\left|\frac{\partial u}{\partial t}\right|$ becoming unbounded as $t\rightarrow 0^+$. The presence of weak initial singularity poses substantial challenges, both in practical and theoretical aspects for traditional numerical techniques due to their inability to accurately capture the behavior of the solution in the vicinity of singular points. Therefore, devising efficient numerical methods that effectively address the singularity at $t = 0$ emerges as a captivating and challenging endeavor.

There exists a wide range of numerical schemes that have been put forth to address the complexities associated with solving the two-dimensional time-fractional convection-diffusion equations.
Cui \cite{cui2015compact} made a significant contribution by introducing a compact exponential ADI scheme that solves two-dimensional TFCD equation numerically.
Chen and Liu \cite{chen2008adi} proposed a novel technique for solving the two-dimensional fractional advection-dispersion equation, employing a combination of the ADI-Euler method and extrapolation method.  Wu and Zhai \cite{wu2020new} utilized Pad$\acute{e}$ approximation to develop a high-order finite difference scheme for solving the two-dimensional TFCD equation. The authors \cite{wang2016error} solved the two-dimensional TFCD equation by employing a combination of the L-1 method of order $2-\alpha$ for temporal discretization and a compact scheme in the spatial domain. Also, the work in \cite{tayebi2017meshless} introduces a meshless approach to tackle the two-dimensional time fractional advection-diffusion equation. Additionally, a second-order scheme for the 2D time TFCD equation was introduced by Zhang et al. in their work \cite{zhang2018approximation}.
The aforementioned schemes, being formulated on uniform meshes, may fall short in adequately addressing the singularity present in the solution at $t=0$. Consequently, there arises a critical need to devise a numerical scheme that can function optimally in a realistic scenario where the solution may not possess adequate smoothness. 

The utilization of fitted meshes has emerged as a powerful tool for effectively tackling the singularity exhibited by the solution. Strategically adapting the fitted meshes provides a means to accurately capture the behavior of the solution near $t=0$ (initial singularity), enabling more accurate and reliable numerical computations. This powerful technique has revolutionized the field by enhancing the accuracy and efficiency of numerical methods in handling challenging problems with singularities \cite{singh2023fully,stynes2017error,chen20191,chen2021two,cen2018numerical,chen2019error}. Even though nonuniform graded meshes have been seen in some works \cite{roul2022high,qiao2021fast}, the temporal convergence order achieved remains somewhat low. More importantly, we are not aware of any numerical method that gives high orders of convergence both in space and time for problem  (\ref{1})-(\ref{2}) with weak initial singularities.


In light of these considerations, the primary objective of this study is to develop and analyze a high-order fully discrete method that is specifically tailored to tackle the complexities associated with the two-dimensional time-fractional convection-diffusion equation, particularly when confronted with weak initial singularities. The proposed numerical approach combines Alikhanov's high-order scheme (L2-1$_\sigma$) for the discretization of the Caputo time-fractional derivative on a non-uniform fitted mesh, along with a high-order two-dimensional compact difference operator on a uniform mesh for spatial discretization. The resulting system is solved using a robust two-step Alternating Direction Implicit (ADI) approach. 
The comprehensive theoretical analysis, focusing on convergence and stability, is rigorously performed using the widely recognized Fourier analysis method. Importantly, it is worth highlighting that this specific scheme, as detailed herein, represents a novel contribution, as no prior literature has explored its applicability for numerically approximating the problem described by equations (\ref{1})-(\ref{2}).

Notation: In this paper, we utilize the symbol $c$ to represent a constant independent of the discretization parameters in space and time. Further, $c$ can assume distinct values depending on the specific context.

The remaining paper follows a structured approach, beginning with the development of the numerical method for the two-dimensional TFCD equation in Section \ref{sec:2}. The subsequent section, Section \ref{sec:3}, focuses on the stability and convergence analysis of the method. The final section, Section \ref{con}, concludes the article by summarizing the key findings.

\section{A high order fully discrete numerical scheme}\label{sec:2}
The following section introduces a novel high-order numerical technique for solving the two-dimensional time-fractional convection-diffusion (TFCD) equation (\ref{1})-(\ref{2}). To uphold fourth-order accuracy in space and the tridiagonal property of the method, we introduce a transformation as a prerequisite to eliminate the convection terms in the equation. Notably, this transformation is similar to the one detailed in \cite{zhai2014unconditionally}. Let
\begin{align}\label{6}
	u(x,y,t)=\frac{v(x,y,t)}{\mathcal P(x)\mathcal Q(y)},
\end{align}
where $\mathcal P(x)=\exp\left(\frac{\mu_1}{2\lambda_1}x\right)$ and $\mathcal Q(y)=\exp\left(\frac{\mu_2}{2\lambda_2}y\right)$.

By using equations (\ref{1})-(\ref{2}) and (\ref{6}), we derive the subsequent time-fractional problem, wherein $v = v(x, y, t)$ serves as the solution: 
\begin{align}\label{7}
	\partial^{\alpha}_{t}v=\lambda_1 v_{xx}+\lambda_2 v_{yy}-\beta v+F,~~  ~\mbox{in}~\Omega \times (0,T_f],~\Omega \subset \mathds{R}^2,
\end{align}
with initial and boundary conditions
\begin{equation}\label{8}
	\left\{%
	\begin{array}{ll}
		v(x,y,0)=\tilde \phi(x,y),&(x,y)\in \bar{\Omega}=\Omega\cup \partial\Omega,\\
		v(x,y,t)=\tilde \psi(x,y,t),&({x},y)\in \partial\Omega,~0<t\leq T_f,
	\end{array}%
	\right.
\end{equation}
where $\beta=\frac{\mu_1^2}{4\lambda_1}+\frac{\mu_2^2}{4\lambda_2}-\gamma$, $F=\mathcal P(x)\mathcal Q(y)f$, $\tilde \phi=\mathcal P(x)\mathcal Q(y)\phi$, and $\tilde \psi=\mathcal P(x)\mathcal Q(y)\psi$.

Here, the direct implication is evident: $u(x, y, t)$ is a solution of (\ref{1})-(\ref{2}) if and only if $v(x, y, t) = \mathcal P(x)\mathcal Q(y)u(x, y, t)$ is a solution of (\ref{7})-(\ref{8}). Further, the subsequent step entails a detailed examination solely focused on problem (\ref{7})-(\ref{8}).

\subsection{\bf{High-order discretization of time-fractional derivative}}
Suppose $N_t$ and $\hat {N_t}$ are positive integers, satisfying the conditions $\hat {N_t}\leq N_t$ and $\hat {N_t}\geq cN_t$, where $c$ is a fixed constant in $(0,1)$. We partition the interval $[0, T_f]$ into two subintervals $[0, T]$ and $[T, T_f]$, where $T\in (0,T_f]$ is arbitrary. To divide $[0,T]$, we employ a graded mesh with mesh points denoted by $t_i = T\left( \frac{i}{\hat {N_t}} \right)^{\theta}$ for $0\leq i\leq \hat {N_t}$. The user has the flexibility to select the grading parameter $\theta$ subject to the constraint $\theta\geq1$. Further, for ease of implementation, we consider the quasiuniform mesh on the interval $[T, T_f]$ (i.e., the ratio $\max_i \tau_i/ \min_i \tau_i$ is always kept within a constant limit), with mesh points $T= t_{\hat {N_t}} < t_{\hat {N_t}+1} < \dots < t_{{N_t}-1} < t_{N_t} = T_f$. The term \enquote{fitted} is used to describe this mesh because, through the proper choice of the grading parameter ${\theta}$, the graded section of the mesh adeptly manages functions with a weak singularity at $t = 0$.

Further, set $\tau_i=t_i-t_{i-1}$, for $1\leq i\leq {N_t}$ and $t_{i+\sigma}=t_i+\sigma\tau_{i+1}$, for $0\leq i\leq {N_t-1}$.
By employing Alikhanov’s formula (L2-1$_\sigma$), we approximate the Caputo derivative of function $g\in C[0,T_f]\cap C^3(0,T_f]$ at $t_{i+\sigma}$ given in \cite{chen2019error}
\begin{align}\label{9}
	\partial^{\alpha}_{t}g(t_{i+\sigma})= w^{(\alpha,\sigma)}_{i,i}g(t_{i+1})-\sum_{k=1}^{i}\left(w^{(\alpha,\sigma)}_{i,k}-w^{(\alpha,\sigma)}_{i,k-1}\right)g(t_{k})-w^{(\alpha,\sigma)}_{i,0}g(t_0)+(\mathcal R_t^{\alpha})^{i+\sigma},
\end{align}
where $(\mathcal R_t^{\alpha})^{i+\sigma}$ is the local truncation error term and the coefficients $w^{(\alpha,\sigma)}$'s are given as: $w^{(\alpha,\sigma)}_{0,0}=\tau_1^{-1}r^{(\alpha,\sigma)}_{0,0},$ and for $i\geq1,$
\begin{equation}\label{10}
	\displaystyle w^{(\alpha,\sigma)}_{i,k}= \begin{cases} \tau_{k+1}^{-1}\left(r^{(\alpha,\sigma)}_{i,0}+s^{(\alpha,\sigma)}_{i,0}\right), & k=0,\vspace{0.1cm}\\
		\tau_{k+1}^{-1}\left(r^{(\alpha,\sigma)}_{i,k}+s^{(\alpha,\sigma)}_{i,k-1}-s^{(\alpha,\sigma)}_{i,k}\right), & 1\leq k\leq i-1,\vspace{0.1cm}\\
		\tau_{k+1}^{-1}\left(r^{(\alpha,\sigma)}_{i,i}+s^{(\alpha,\sigma)}_{i,i-1}\right), & k=i,
	\end{cases}
\end{equation}
with  \begin{align}\label{11}
	&r^{(\alpha,\sigma)}_{i,i}=\frac{\sigma^{1-\alpha}}{\Gamma(2-\alpha)}\tau_{i+1}^{1-\alpha},~\text{for}~ i\geq 0,	
\end{align}
and for $i\geq1$, $0\leq k\leq i-1,$
\begin{align}
	r^{(\alpha,\sigma)}_{i,k}=&\frac{1}{\Gamma(1-\alpha)}\int_{t_k}^{t_{k+1}}(t_{i+\sigma}-\nu)^{-\alpha}d\nu,\label{12}\\
	s^{(\alpha,\sigma)}_{i,k}=&\frac{1}{\Gamma(1-\alpha)}\frac{2}{(t_{k+2}-t_k)}\int_{t_k}^{t_{k+1}}(t_{i+\sigma}-\nu)^{-\alpha}(\nu-t_{k+1/2})d\nu.\label{13}
\end{align}
\begin{lemma}\label{L:1}
	\emph{\cite{chen2019error}}	Suppose $1-{\alpha}/{2}\leq\sigma\leq1$ and the local mesh ratio $\eta_i = \frac{\tau_{i+1}}{\tau_i}$ for $1\leq i\leq N_t-1$ satisfies $3/4 \leq\eta_i\leq62.$ Then
	\begin{align*}
		&\emph{(i).}~w^{(\alpha,\sigma)}_{i,0}>\frac{t_{i+\sigma}^{-\alpha}}{\Gamma(1-\alpha)}>0,~i\geq 0.~~~~~~~~~~~~~~~~~~~~~~~~~~~~~~~~~~~~~~~~~~~~~~~~~~~~~~~~~~~~~~~~~~~~~~{}\\
		&	\emph{(ii).}~(2\sigma-1)w^{(\alpha,\sigma)}_{1,1}-\sigma w^{(\alpha,\sigma)}_{1,0}>0.\\
		&	\emph{(iii).}~w^{(\alpha,\sigma)}_{i,1}>w^{(\alpha,\sigma)}_{i,0},~~i\geq 1.\\
		&	\emph{(iv).}~\mbox{If} ~\eta_{k-1}^2(\eta_{k-1}+1) \geq \frac{\eta_{k}}{\eta_{k}+1}~for~2\leq k\leq i,~with~i\geq2,~then~ w^{(\alpha,\sigma)}_{i,k-1}<w^{(\alpha,\sigma)}_{i,k}.\\
			&	\emph{(v).}~\mbox{If} ~\eta_{i-1}^2\left(2-\frac{1}{\sigma}+\eta_{i}(\eta_{i}+2)\right) \geq \frac{\eta_{i}(\eta_{i}+1)}{\eta_{i-1}+1}~for~2\leq i\leq N_t,~then~ (2\sigma-1)w^{(\alpha,\sigma)}_{i,i}-\sigma w^{(\alpha,\sigma)}_{i,i-1}>0.
	\end{align*}
\end{lemma}	
At point $t=t_{i+\sigma},$ equation (\ref{7}) takes the form
\begin{align}\label{14}
	\nonumber	\partial^{\alpha}_{t}v^{i+\sigma}(x,y)=\lambda_1 v_{xx}^{i+\sigma}(x,y)+\lambda_2 v_{yy}^{i+\sigma}&(x,y)-\beta v^{i+\sigma}(x,y)+F^{i+\sigma}(x,y),\\& ~~({x},y)\in \Omega,~ i=0,1,\dots,N_t-1,
\end{align}
with initial and boundary conditions
\begin{equation}\label{15}
	\left\{%
	\begin{array}{ll}
		u^{0}(x,y)=\tilde\phi(x,y),&(x,y)\in \bar{\Omega}=\Omega\cup \partial\Omega,\\
		u^i(x,y)=\tilde\psi(x,y,t_i),&({x},y)\in \partial\Omega, i=1,2,\dots,N_t,
	\end{array}%
	\right.
\end{equation}
where we denote $u^{i}(x,y)=u(x,y,t_i)$.

Employing the L2-1$_\sigma$ approximation on a non-uniform mesh, as described in equation (\ref{9}), and incorporating it into equation (\ref{14}), yields the following semi discrete scheme
\begin{align}\label{16}
	\nonumber	w^{(\alpha,\sigma)}_{i,i}&v^{i+1}(x,y)-\sum_{k=1}^{i}\left(w^{(\alpha,\sigma)}_{i,k}-w^{(\alpha,\sigma)}_{i,k-1}\right)v^{k}(x,y)-w^{(\alpha,\sigma)}_{i,0}v^{0}(x,y)+(\mathcal R_t^{\alpha})^{i+\sigma}=\lambda_1 v_{xx}^{i+\sigma}(x,y)\\&+\lambda_2 v_{yy}^{i+\sigma}(x,y)-\beta v^{i+\sigma}(x,y)+F^{i+\sigma}(x,y),~~ ~({x},y)=:\Omega,~ i=0,1,\dots,N_t-1.
\end{align}
\subsection{\bf{High-order compact alternating direction implicit (ADI) scheme}}

Now, our subsequent strategy for solving the two-dimensional problem (\ref{16}) with initial and boundary conditions (\ref{15}) entails the discretization of the spatial variables through the utilization of a compact finite difference scheme. The compact finite difference method is a well-established numerical technique extensively utilized in solving PDEs. It has been studied and applied in various scientific and engineering fields. Numerous research works have demonstrated the effectiveness and advantages of the high-order compact finite difference method \cite{dablain1986application,liao2010error, cui2009compact, gao2011compact, ran2018new, sumit2022optimal, ferras2021high,kumar2022high,singhhigh}. For given positive integers $M_x$ and $M_y$, we discretize the problem using $x_m = mh_x$ and 
$y_n = nh_y$, where $0\leq m\leq M_x$ and 
$0\leq n\leq  M_y$. 
Here,  $h_x=\frac{L}{M_x}$ is the step size in the $x$ direction, and $h_y=\frac{L}{M_y}$ is the step size in the $y$ direction.	Let $v=\{v^i_{m,n}~|~ 0\leq m\leq M_x, 0\leq n\leq  M_y,~\text{and}~ 
1\leq i\leq  N_t$\}. We introduce spatial difference operators
\begin{equation}\label{17}
	\displaystyle  \mathcal H_xv^i_{m,n}= \begin{cases} \left(I+\frac{h_x^{2}}{12}\delta_{{x}}^{2}\right)v_{m,n}^{i}, & 1\leq m \leq M_{{x}}-1,\\
		v^i_{m,n}, & m=0\hspace{0.1cm} \text{or} \hspace{0.1cm} M_{x},
	\end{cases}
\end{equation}

\begin{equation}\label{18}
	\displaystyle\mathcal H_yv^i_{m,n}= \begin{cases} \left(I+\frac{h_y^{2}}{12}\delta_{{y}}^{2}\right)v_{m,n}^{i}, & 1\leq n \leq M_{y}-1,\\
		v^i_{m,n}, & n=0\hspace{0.1cm} \text{or} \hspace{0.1cm} M_{y},
	\end{cases}
\end{equation}
where
\begin{align}\label{19}
	\delta_{{x}}^{2}v^i_{m,n}=\frac{v^i_{m+1,n}-2v^i_{m,n}+v^i_{m-1,n}}{h_x^{2}}~~\text{and}~~	\delta_{y}^{2}v^i_{m,n}=\frac{v^i_{m,n+1}-2v^i_{m,n}+v^i_{m,n-1}}{h_y^{2}}.
\end{align}
Now, re-write equation (\ref{16}) at $(x_m,y_n)$ in the following form

\begin{align}\label{20}
	-\lambda_1 v_{xx}^{i+\sigma}(x_m,y_n)-\lambda_2 v_{yy}^{i+\sigma}(x_m,y_n)=G^{i+\sigma}(x_m,y_n)+(\mathcal R_t^{\alpha})^{i+\sigma},
\end{align}
where
\begin{align}\label{21}
	\nonumber G^{i+\sigma}(x_m,y_n)=&~F^{i+\sigma}(x_m,y_n)-w^{(\alpha,\sigma)}_{i,i}v^{i+1}(x_m,y_n)+\sum_{k=1}^{i}\left(w^{(\alpha,\sigma)}_{i,k}-w^{(\alpha,\sigma)}_{i,k-1}\right)v^{k}(x_m,y_n)\\&~~+w^{(\alpha,\sigma)}_{i,0}v^{0}(x_m,y_n)-\beta v^{i+\sigma}(x_m,y_n).
\end{align}
In accordance with the techniques presented in \cite{liao2010maximum,zhang2011alternating}, we develop a fourth-order compact scheme for equation (\ref{20}) as follows
\begin{align}\label{22}
	-\lambda_1\mathcal H_x^{-1}	\delta_{{x}}^{2}v^{i+\sigma}_{m,n}-\lambda_2   \mathcal H_y^{-1}	\delta_{y}^{2}v^{i+\sigma}_{m,n}=G^{i+\sigma}_{m,n}+(\mathcal R_t^{\alpha})^{i+\sigma}+(\mathcal R_l)^{i+\sigma}_{m,n},
\end{align}
where 
\begin{align*}
	(\mathcal R_l)^{i+\sigma}_{m,n}=(\mathcal R_x)^{i+\sigma}_{m,n}+(\mathcal R_y)^{i+\sigma}_{m,n},
\end{align*}
with $\varphi(\eta)=-3(1-\eta)^5+5(1-\eta)^3,$ and
\begin{align*}
	&(\mathcal R_x)^{i+\sigma}_{m,n}=\frac{h_x^4}{360}\int_0^1\left(  \frac{\partial^6v}{\partial x^6}(x_m-\eta h_x,y_n,t_{i+\sigma})+\frac{\partial^6v}{\partial x^6}(x_m+\eta h_x,y_n,t_{i+\sigma})\right)\varphi(\eta)d\eta,\\
	&(\mathcal R_y)^{i+\sigma}_{m,n}=\frac{h_y^4}{360}\int_0^1\left(  \frac{\partial^6v}{\partial y^6}(x_m,y_n-\eta h_y,t_{i+\sigma})+\frac{\partial^6v}{\partial y^6}(x_m,y_n+\eta h_y,t_{i+\sigma})\right)\varphi(\eta)d\eta.
\end{align*}
Thus, we have 
\begin{align}\label{23}
	(\mathcal R_l)^{i+\sigma}_{m,n}\leq c~ (h_x^4+h_y^4).
\end{align}
By multiplying both sides of equation (\ref{22}) with the operator $\mathcal H_x \mathcal H_y$, we obtain
\begin{align}\label{24}
	(-\lambda_1  \mathcal H_y	\delta_{{x}}^{2}-\lambda_2   \mathcal H_x	\delta_{y}^{2})v^{i+\sigma}_{m,n}= \mathcal H_x \mathcal H_yG^{i+\sigma}_{m,n}+(\mathcal R_t^{\alpha})^{i+\sigma}+(\mathcal R_l)^{i+\sigma}_{m,n}.
\end{align}

\begin{lemma}\label{L:2}
	\emph{\cite{chen2019error}}	For any function $g(t) \in C^2(0, T_f],$ one has
	\begin{align*}
		\left|g(t_{i+\sigma})-\{\sigma g(t_{i+1})+(1-\sigma) g(t_{i})\}\right|\leq \frac{\tau_{i+1}^2}{8}\mathop {\max}_{1\leq i \leq N_{t}}|g^{\prime\prime}(t_i)|.
	\end{align*}
\end{lemma}	

Therefore, by utilizing Lemma \ref{L:2} along with equation (\ref{21}), into equation (\ref{24}), we obtain the following expression
\begin{align*}
	&\mu\left( \mathcal H_x \mathcal H_y-\frac{\sigma\lambda_1}{\mu} \mathcal H_y	\delta_{{x}}^{2}-\frac{\sigma\lambda_2}{\mu}  \mathcal H_x	\delta_{y}^{2}\right)v^{i+1}_{m,n}=\sum_{k=1}^{i}\left(w^{(\alpha,\sigma)}_{i,k}-w^{(\alpha,\sigma)}_{i,k-1}\right)\mathcal H_x\mathcal H_yv^{k}_{m,n}+w^{(\alpha,\sigma)}_{i,0}\mathcal H_x\mathcal H_yv^{0}_{m,n}\\&-\beta(1-\sigma) \mathcal H_x\mathcal H_yv^{i}_{m,n}+(1-\sigma)(\lambda_1 \mathcal H_y	\delta_{{x}}^{2}+\lambda_2  \mathcal H_x	\delta_{y}^{2})v^{i}_{m,n}+\mathcal H_x\mathcal H_yF^{i+\sigma}_{m,n}+(\mathcal R_t^{\alpha})^{i+\sigma}+(\mathcal R_l)^{i+\sigma}_{m,n}+\mathcal R_t^{i+1},
\end{align*}
where, $\mu=\beta\sigma+w^{(\alpha,\sigma)}_{i,i}>0$. Also, by invoking Lemma \ref{L:2}, we can obtain a bound for $\mathcal R_t^{i+1}$.

After re-arranging the terms of above equation, we get
\begin{align}\label{25}
	\nonumber	&~~~\left(\mathcal H_x\mathcal H_y-\frac{\sigma\lambda_1}{\mu} \mathcal H_y	\delta_{{x}}^{2}-\frac{\sigma\lambda_2}{\mu}  \mathcal H_x	\delta_{y}^{2}\right)v^{i+1}_{m,n}=\frac{1}{\mu}\sum_{k=1}^{i}\left(w^{(\alpha,\sigma)}_{i,k}-w^{(\alpha,\sigma)}_{i,k-1}\right)\mathcal H_x\mathcal H_yv^{k}_{m,n}\\
	\nonumber&~~~~~~~+\frac{w^{(\alpha,\sigma)}_{i,0}}{\mu}\mathcal H_x\mathcal H_yv^{0}_{m,n}-\frac{\beta(1-\sigma)}{\mu}  \mathcal H_x\mathcal H_yv^{i}_{m,n}+\frac{(1-\sigma)}{\mu}\left(\lambda_1 \mathcal H_y	\delta_{{x}}^{2}+\lambda_2  \mathcal H_x	\delta_{y}^{2}\right)v^{i}_{m,n}\\&~~~~~~~~~~~+\frac{1}{\mu}\mathcal H_x\mathcal H_yF^{i+\sigma}_{m,n}+\frac{(\mathcal R_t^{\alpha})^{i+\sigma}}{\mu}+\frac{(\mathcal R_l)^{i+\sigma}_{m,n}}{\mu}+\frac{\mathcal R_t^{i+1}}{\mu}.
\end{align}
Further, to devise an efficient ADI scheme, we incorporate the following perturbation term 
\begin{align}\label{26}
	\frac{\lambda_1\lambda_2\sigma^2}{\mu^2}\delta_{x}^{2}\delta_{y}^{2}(v^{i+1}_{m,n}-v^{i}_{m,n})=(\mathcal R_p)_{m,n}^{i+1}.
\end{align}
Now, adding this perturbation term in equation (\ref{25}), we get
\begin{align}\label{27}
	\nonumber&~~~\left(\mathcal H_x-\frac{\sigma\lambda_1}{\mu}\delta_{x}^{2}\right)\left(\mathcal H_y-\frac{\sigma\lambda_2}{\mu}\delta_{y}^{2}\right)v^{i+1}_{m,n}=\frac{1}{\mu}\sum_{k=1}^{i}\left(w^{(\alpha,\sigma)}_{i,k}-w^{(\alpha,\sigma)}_{i,k-1}\right)\mathcal H_x\mathcal H_yv^{k}_{m,n}\\\nonumber&~~~~~+\frac{w^{(\alpha,\sigma)}_{i,0}}{\mu}\mathcal H_x\mathcal H_yv^{0}_{m,n}-\frac{\beta(1-\sigma)}{\mu}  \mathcal H_x\mathcal H_yv^{i}_{m,n}+\frac{(1-\sigma)}{\mu}\left(\lambda_1 \mathcal H_y	\delta_{{x}}^{2}+\lambda_2  \mathcal H_x	\delta_{y}^{2}\right)v^{i}_{m,n}\\&~~~~~~~~~~~+\frac{1}{\mu}\mathcal H_x\mathcal H_yF^{i+\sigma}_{m,n}+\frac{\lambda_1\lambda_2\sigma^2}{\mu^2}\delta_{x}^{2}\delta_{y}^{2}v^{i}_{m,n}+ (\mathcal R_f)_{m,n}^{i+1},
\end{align}
where
\begin{align}\label{28}
	(\mathcal R_f)_{m,n}^{i+1}=\frac{(\mathcal R_t^{\alpha})^{i+\sigma}}{\mu}+\frac{(\mathcal R_l)^{i+\sigma}_{m,n}}{\mu}+\frac{\mathcal R_t^{i+1}}{\mu}+(\mathcal R_p)_{m,n}^{i+1}.
\end{align}
By omitting the error term in equation (\ref{27}), we derive the following fully discrete numerical scheme
\begin{align}\label{29}
	\nonumber&~~~\left(\mathcal H_x-\frac{\sigma\lambda_1}{\mu}\delta_{x}^{2}\right)\left(\mathcal H_y-\frac{\sigma\lambda_2}{\mu}\delta_{y}^{2}\right)V^{i+1}_{m,n}=\frac{1}{\mu}\sum_{k=1}^{i}\left(w^{(\alpha,\sigma)}_{i,k}-w^{(\alpha,\sigma)}_{i,k-1}\right)\mathcal H_x\mathcal H_yV^{k}_{m,n}\\\nonumber&~~~~~~~+\frac{w^{(\alpha,\sigma)}_{i,0}}{\mu}\mathcal H_x\mathcal H_yV^{0}_{m,n}-\frac{\beta(1-\sigma)}{\mu}  \mathcal H_x\mathcal H_yV^{i}_{m,n}+\frac{(1-\sigma)}{\mu}\left(\lambda_1 \mathcal H_y	\delta_{{x}}^{2}+\lambda_2  \mathcal H_x	\delta_{y}^{2}\right)V^{i}_{m,n}\\&~~~~~~~~~~~~~~~~+\frac{1}{\mu}\mathcal H_x\mathcal H_yF^{i+\sigma}_{m,n}+\frac{\lambda_1\lambda_2\sigma^2}{\mu^2}\delta_{x}^{2}\delta_{y}^{2}V^{i}_{m,n},
\end{align}
where $V^{i}_{m,n}$ represents the numerical approximation of $v^{i}_{m,n}$.

Further, to enhance the ease of computations, we incorporate the intermediate variable $V_{m,n}^*$ defined by
\begin{align}\label{30}
	V_{m,n}^*=\left(\mathcal H_y-\frac{\sigma\lambda_2}{\mu}\delta_{y}^{2}\right)V^{i+1}_{m,n},~0\leq m \leq M_x,~1 \leq n \leq M_y-1.
\end{align}
Now, by employing a sequential two-step procedure, equation (\ref{29}) can be effectively solved as described below:

\noindent \textbf{Step (i):}  To determine $\{V_{m,n}^*\}$, we solve the following set of linear equations, for a fixed $n\in\{1,2,\dots M_y-1\},$
\begin{align}\label{31}
	\nonumber&\left(\mathcal H_x-\frac{\sigma\lambda_1}{\mu}\delta_{x}^{2}\right)	V_{m,n}^*=\frac{1}{\mu}\sum_{k=1}^{i}\left(w^{(\alpha,\sigma)}_{i,k}-w^{(\alpha,\sigma)}_{i,k-1}\right)\mathcal H_x\mathcal H_yV^{k}_{m,n}+\frac{w^{(\alpha,\sigma)}_{i,0}}{\mu}\mathcal H_x\mathcal H_yV^{0}_{m,n}\\\nonumber&~~-\frac{\beta(1-\sigma)}{\mu}  \mathcal H_x\mathcal H_yV^{i}_{m,n}+\frac{(1-\sigma)}{\mu}\left(\lambda_1 \mathcal H_y	\delta_{{x}}^{2}+\lambda_2  \mathcal H_x	\delta_{y}^{2}\right)V^{i}_{m,n}+\frac{1}{\mu}\mathcal H_x\mathcal H_yF^{i+\sigma}_{m,n}\\&~~~~+\frac{\lambda_1\lambda_2\sigma^2}{\mu^2}\delta_{x}^{2}\delta_{y}^{2}V^{i}_{m,n},~~1 \leq m\leq M_x-1,
\end{align}
with boundary conditions
\begin{align}\label{32}
	V_{0,n}^*=\left(\mathcal H_y-\frac{\sigma\lambda_2}{\mu}\delta_{y}^{2}\right)V^{i+1}_{0,n},~~V_{M_x,n}^*=\left(\mathcal H_y-\frac{\sigma\lambda_2}{\mu}\delta_{y}^{2}\right)V^{i+1}_{M_x,n},
\end{align}
where
\begin{align}\label{33}
	V^{i+1}_{0,n}=\tilde\psi(x_0,y_n,t_{i+1}),~~V^{i+1}_{M_x,n}=\tilde\psi(x_{M_x},y_n,t_{i+1}).
\end{align}

\noindent \textbf{Step (ii):} Once the $\{V_{m,n}^*\}$ values have been determined, we solve the subsequent set of linear equations to obtain the final solution $\{V_{m,n}^{i+1}\}$, for a fixed $m\in\{1,2,\dots M_x-1\},$ as follows
\begin{align}\label{34}
	\left(\mathcal H_y-\frac{\sigma\lambda_2}{\mu}\delta_{y}^{2}\right)V^{i+1}_{m,n}=V_{m,n}^*,~~1\leq n\leq M_y-1,
\end{align}
with boundary conditions
\begin{align}\label{35}
	V^{i+1}_{m,0}=\tilde\psi(x_m,y_0,t_{i+1}),~~V^{i+1}_{m,M_y}=\tilde\psi(x_m,y_{M_y},t_{i+1}).
\end{align}

\section{Theoretical analysis}\label{sec:3}
\noindent	We will now explore the truncation error, stability, and convergence aspects of the developed scheme in this section.

\subsection{\bf{Truncation error}}
Now, we proceed to estimate the truncation error, denoted as $(\mathcal R_f)_{m,n}^{i+1}$, of the scheme (\ref{29}).	

\begin{lemma}\label{L:3}
	Let $\sigma=1-\frac{\alpha}{2}.$ Assume that the local mesh ratio $\eta_i = \frac{\tau_{i+1}}{\tau_i}$ for $1\leq i\leq N_t-1$ satisfies $3/4 \leq\eta_i\leq62$. Then
	\begin{align*}
		\frac{1}{w^{(\alpha,\sigma)}_{i,i}}=\mathcal{O}(\tau_{i+1}^{\alpha}),~for~ 0\leq i\leq N_t-1.
	\end{align*}
\end{lemma}	
\begin{proof}
	It is obvious for $i=0$ that 
	\begin{align*}
		w^{(\alpha,\sigma)}_{0,0}=\frac{\sigma^{1-\alpha}}{\Gamma(2-\alpha)}\frac{1}{\tau_1^{\alpha}},
	\end{align*}
	which imply
	\begin{align}\label{3.1}
		\left|\frac{1}{w^{(\alpha,\sigma)}_{0,0}}\right|\leq\frac{\Gamma(2-\alpha)}{\sigma^{1-\alpha}}{\tau_1^{\alpha}}.
	\end{align}
	For $i\geq1,$
	\begin{align*}
		w^{(\alpha,\sigma)}_{i,i}=&\frac{1}{\tau_{i+1}}\left(r^{(\alpha,\sigma)}_{i,i}+s^{(\alpha,\sigma)}_{i,i-1}\right)=\frac{1}{\tau_{i+1}^{\alpha}}\left[\frac{\sigma^{1-\alpha}}{\Gamma(2-\alpha)}+\frac{s^{(\alpha,\sigma)}_{i,i-1}}{\tau_{i+1}^{1-\alpha}}\right],
	\end{align*}
	which on solving yields
	\begin{align}\label{3.2}
		w^{(\alpha,\sigma)}_{i,i}=&\frac{\sigma^{1-\alpha}}{\Gamma(2-\alpha)}\frac{\varrho}{\tau_{i+1}^{\alpha}},
	\end{align}
	where
	\begin{align*}
		\varrho=1+\left(1+\frac{1}{\eta_i}\right)^{-1}\left[         \Biggl\{\left(1+\frac{1}{\sigma\eta_i}\right)^{2-\alpha}-1\Biggr\}-\frac{1}{\eta_i}\Biggl\{\left(1+\frac{1}{\sigma\eta_i}\right)^{1-\alpha}+1\Biggr\}\right],
	\end{align*}
with $\eta_i = \frac{\tau_{i+1}}{\tau_i}$. Further, by employing the bounds of $\eta_i$, where $1\leq i\leq N_t-1$, with values confined within the range $3/4 \leq \eta_i \leq 62$, as established in \cite{chen2019error}, we can derive the corresponding bounds of $\varrho$ as
	\begin{align}\label{3.3}
		1.6597542\leq \varrho\leq 13.215168.
	\end{align}
By utilizing the relation (\ref{3.1}) and incorporating the bounds of $\varrho$ as provided in (\ref{3.3}) into the expression (\ref{3.2}), we obtain the following result
	\begin{align*}
		\left|\frac{1}{w^{(\alpha,\sigma)}_{i,i}}\right|\leq c~\tau_{i+1}^{\alpha},
	\end{align*}
	where $c$ is positive constant.
	
	Thus, we have the Lemma.
\end{proof}
Now we define the following functions
\begin{equation}\label{3.4}
	\left\{
	\begin{array}{lll}
	Q_g^{\sigma}&=\tau_1^{\alpha}\displaystyle\sup_{\varepsilon\in(0,t_1)}\left(\varepsilon^{1-\alpha}|\delta_t g(t_0)-g^{\prime}(\varepsilon)|\right),\\
			Q_g^{i+\sigma}&=t_{i+\sigma}^{\alpha}\tau_{i+1}^{3-\alpha}\displaystyle\sup_{\varepsilon\in(t_i,t_{i+1})}|g^{\prime\prime\prime}(\varepsilon)|,~\text{for}~1\leq i\leq N_t-1,\\
		Q_g^{i,1}&=\tau_1^{\alpha}\displaystyle\sup_{\varepsilon\in(0,t_1)}\left(\varepsilon^{1-\alpha}|(I_{2,1}g(\varepsilon))^{\prime} -g^{\prime}(\varepsilon)|\right),~\text{for}~1\leq i\leq N_t-1,\\
			Q_g^{i,k}&=t_{k}^{\alpha}\tau_{i+1}^{-\alpha}\tau_{k}^2(\tau_{k+1}+\tau_{k})\displaystyle\sup_{\varepsilon\in(t_{k-1},t_{k+1})}|g^{\prime\prime\prime}(\varepsilon)|,~\text{for}~2\leq k\leq i\leq N_t-1,
	\end{array}
	\right.
\end{equation}
for any function $g\in C[0, T] \cap C^3(0, T]$. Here, $\delta_t g(t_0)=\frac{[g(t_1)-g(t_0)]}{\tau_1}$ and $I_{2,1}g(\varepsilon)$ is the quadratic polynomial that interpolates to $g(\varepsilon)$ at the points $t_{0}$, $t_1$, and $t_{2}$.
\begin{lemma}\label{L:4}
	\emph{\cite{chen2019error}} Given the assumptions $\sigma=1-\frac{\alpha}{2},$ $\tau_{i+1}\lesssim t_i$ for $i\geq 2,$ and $\frac{\tau_1}{\tau_2}\leq \rho$ (with $\rho$ being a positive constant), the following result holds true for any function $g(t) \in C^3(0, T]$:
	\begin{align*}
		\left|(\mathcal R_t^{\alpha})^{i+\sigma}\right|\lesssim t_{i+\sigma}^{-\alpha}\left(Q_g^{i+\sigma}+\displaystyle\max_{1\leq k\leq i}\{	Q_g^{i,k}\} \right),~\text{for}~0\leq i\leq N_t-1.
	\end{align*}
\end{lemma}	

\begin{lemma}\label{L:5}
	\emph{\cite{chen2019error}} Suppose $g\in C[0, T] \cap C^3(0, T]$ and satisfies the assumptions stated in (\ref{5}). Then
	\begin{align*}
&Q_g^{i+\sigma}	\lesssim N_t^{-\min\{3-\alpha,\theta\alpha\}},~\text{for}~0\leq i\leq N_t-1,\\
&Q_g^{i,k}	\lesssim N_t^{-\min\{3-\alpha,\theta\alpha\}},~\text{for}~1\leq k\leq i,~\text{when}~ i\geq 1.
	\end{align*}
\end{lemma}	

Thus, combining Lemmas \ref{L:4} and \ref{L:5}, we get
\begin{align}\label{3.5}
	\left|(\mathcal R_t^{\alpha})^{i+\sigma}\right|\lesssim t_{i+\sigma}^{-\alpha} N_t^{-\min\{3-\alpha,\theta\alpha\}}, ~\text{for}~0\leq i\leq N_t-1.
\end{align}
\begin{Theorem}\label{T:3.1}
Suppose that the solution $v(x, y, t)$ of problem (\ref{7})-(\ref{8}) meets the assumptions provided in (\ref{4}) and (\ref{5}). Then, the local truncation error $(\mathcal R_f)_{m,n}^{i+1}$, given in equation (\ref{28}), of the scheme (\ref{29}) satisfies the following bound
\begin{align*}
	|(\mathcal R_f)_{m,n}^{i+1}|\leq c \left(N_t^{-\min\{3-\alpha,\theta\alpha,1+2\alpha,2+\alpha\}}+h_x^4+h_y^4\right).
\end{align*}
	\end{Theorem}
\begin{proof}
	Based on equation (\ref{28}), and $w^{(\alpha,\sigma)}_{i,i}>0$, we deduce that
\begin{align}\label{3.6}
\nonumber	|(\mathcal R_f)_{m,n}^{i+1}|\leq&~\left|\frac{(\mathcal R_t^{\alpha})^{i+\sigma}}{\mu}\right|+\left|\frac{(\mathcal R_l)^{i+\sigma}_{m,n}}{\mu}\right|+\left|\frac{\mathcal R_t^{i+1}}{\mu}\right|+\left|(\mathcal R_p)_{m,n}^{i+1}\right|\\
\leq	&~\frac{|(\mathcal R_t^{\alpha})^{i+\sigma}|}{w^{(\alpha,\sigma)}_{i,i}}+\frac{|(\mathcal R_l)^{i+\sigma}_{m,n}|}{w^{(\alpha,\sigma)}_{i,i}}+\frac{|\mathcal R_t^{i+1}|}{w^{(\alpha,\sigma)}_{i,i}}+\left|(\mathcal R_p)_{m,n}^{i+1}\right|.
\end{align}
Next, we proceed to bound each term individually in inequality (\ref{3.6}).

Utilizing Lemma \ref{L:3} together with inequality (\ref{3.5}), we arrive at 
\begin{align}\label{3.7}
\nonumber	\frac{|(\mathcal R_t^{\alpha})^{i+\sigma}|}{w^{(\alpha,\sigma)}_{i,i}}\leq& ~c~\tau_{i+1}^{\alpha} t_{i+\sigma}^{-\alpha} N_t^{-\min\{3-\alpha,\theta\alpha\}}\\
\nonumber	=&~c~\tau_{i+1}^{\alpha} (t_i+\sigma\tau_{i+1})^{-\alpha} N_t^{-\min\{3-\alpha,\theta\alpha\}}\\
\nonumber	\leq&~c~\tau_{i+1}^{\alpha} (\sigma\tau_{i+1})^{-\alpha} N_t^{-\min\{3-\alpha,\theta\alpha\}}\\
	\leq&~c~ N_t^{-\min\{3-\alpha,\theta\alpha\}}.
\end{align}	
Next, we invoke Lemma \ref{L:3} along with relation (\ref{23}) to obtain
\begin{align}\label{3.8}
\frac{|(\mathcal R_l)^{i+\sigma}_{m,n}|}{w^{(\alpha,\sigma)}_{i,i}}\leq c~ (h_x^4+h_y^4).
\end{align}	
Moving forward, by merging Lemmas \ref{L:2} and \ref{L:3}, we get
\begin{align}\label{3.9}
	\frac{|\mathcal R_t^{i+1}|}{w^{(\alpha,\sigma)}_{i,i}}\leq~c~\tau_{i+1}^{2+\alpha}.
\end{align}
 Moreover, it is evident from \cite{stynes2017error} that
\begin{align}\label{3.10}
	\tau_{i+1}\leq C~T_fN_t^{-\theta}i^{\theta-1}\leq C~T_fN_t^{-\theta}N_t^{\theta-1}\leq \frac{C~T_f}{N_t}.
\end{align}
Hence, by employing equations (\ref{3.9}) and (\ref{3.10}), we deduce that
\begin{align}\label{3.11}
	\frac{|\mathcal R_t^{i+1}|}{w^{(\alpha,\sigma)}_{i,i}}\leq~c~N_t^{-(2+\alpha)}.
\end{align}
Finally, by combining Lemma \ref{L:3} with equation (\ref{26}) and incorporating relation (\ref{3.10}), we obtain
\begin{align}\label{3.12}
|(\mathcal R_p)_{m,n}^{i+1}|\leq~c~N_t^{-(1+2\alpha)}.
\end{align}
Therefore, by substituting equations (\ref{3.7}), (\ref{3.8}), (\ref{3.11}), and (\ref{3.12}) into relation (\ref{3.6}), we have the theorem.
\end{proof}

\subsection{\bf{Stability analysis}}

\noindent Let the scheme (\ref{29}) having a perturbed solution $\{\widetilde V_{m,n}^{i}, 1\leq m\leq M_x-1,~1\leq n\leq M_y-1,~1\leq i\leq N_t\}$. Then, we will examine how the
perturbation $\zeta_{m,n}^{i}= V_{m,n}^{i}-\widetilde V_{m,n}^{i}$, $1\leq m\leq M_x-1,~1\leq n\leq M_y-1,~1\leq i\leq N_t$, evolves over time. Note that $\zeta_{m,n}^{i}$ satisfies the following error equation
\begin{align}\label{3.13}
	\nonumber&\left(\mathcal H_x-\frac{\sigma\lambda_1}{\mu}\delta_{x}^{2}\right)\left(\mathcal H_y-\frac{\sigma\lambda_2}{\mu}\delta_{y}^{2}\right)\zeta^{i+1}_{m,n}=\frac{1}{\mu}\sum_{k=1}^{i}\left(w^{(\alpha,\sigma)}_{i,k}-w^{(\alpha,\sigma)}_{i,k-1}\right)\mathcal H_x\mathcal H_y\zeta^{k}_{m,n}\\\nonumber&+\frac{w^{(\alpha,\sigma)}_{i,0}}{\mu}\mathcal H_x\mathcal H_y\zeta^{0}_{m,n}-\frac{\beta(1-\sigma)}{\mu}  \mathcal H_x\mathcal H_y\zeta^{i}_{m,n}+\frac{(1-\sigma)}{\mu}\left(\lambda_1 \mathcal H_y	\delta_{{x}}^{2}+\lambda_2  \mathcal H_x	\delta_{y}^{2}\right)\zeta^{i}_{m,n}\\&+\frac{\lambda_1\lambda_2\sigma^2}{\mu^2}\delta_{x}^{2}\delta_{y}^{2}\zeta^{i}_{m,n},~~1\leq m\leq M_x-1,~1\leq n\leq M_y-1,~0\leq i\leq N_t-1,
\end{align}
where 

$\zeta^{i+1}= [\zeta^{i+1}_{1,1},\zeta^{i+1}_{1,2},\dots,\zeta^{i+1}_{1,M_y-1},\zeta^{i+1}_{2,1},\zeta^{i+1}_{2,2},\dots,\zeta^{i+1}_{2,M_y-1},\dots,\zeta^{i+1}_{M_x-1,1},\zeta^{i+1}_{M_x-1,2},\dots,\zeta^{i+1}_{M_x-1,M_y-1}]^T.$

 The grid function $\zeta^{i+1}(x,y),~0\leq i\leq N_t-1$, is defined as 
\begin{align}\label{3.14}
	\zeta^{i+1}(x,y)=
	\begin{cases}
		\zeta^{i+1}_{m,n}, & x_{m-\frac{1}{2}}<x\leq x_{m+\frac{1}{2}}~, y_{n-\frac{1}{2}}<y\leq y_{n+\frac{1}{2}},\\
		& 1\leq m\leq M_x-1~,1\leq n\leq M_y-1, \\
		0, &0\leq x\leq \frac{h_x}{2}~,~L-\frac{h_x}{2}<x\leq L,\\
		& 0\leq y\leq \frac{h_y}{2}~,~L-\frac{h_y}{2}<y\leq L.
	\end{cases}
\end{align}

The expression of $\zeta^{i+1}(x,y)$ as a Fourier series is as follows
\begin{align}\label{3.15}
	\zeta^{i+1}(x,y)=\mathop{\sum^{\infty}\sum^{\infty}}_{k_1=-\infty\ k_2=-\infty}\omega^{i+1}(k_1,k_2)e^{2\pi\iota\left(\frac{k_1x+k_2y}{L}\right)},~0\leq i\leq N_t-1,
\end{align}
where \begin{align}\label{3.16}
	\omega^{i+1}(k_1,k_2)=\frac{1}{L^2}\int_0^L\int_0^L\zeta^{i+1}(x,y)e^{-2\pi\iota\left(\frac{k_1x+k_2y}{L}\right)}dxdy.
\end{align}
By virtue of the $L_2$-discrete norm definition and Parseval's equality, we have
\begin{align}\label{3.17}
	\|\zeta^{i+1}\|^2_2=\mathop{\sum^{M_x-1}\sum^{M_y-1}}_{m=1\ n=1} h_xh_y|\zeta^{i+1}_{m,n}|^2=L^2\mathop{\sum^{\infty}\sum^{\infty}}_{k_1=-\infty\ k_2=-\infty}|\omega^{i+1}(k_1,k_2)|^2.
\end{align}
Let us consider the form of the solution of equation (\ref{3.13}) to be 
\begin{align}\label{3.18}
	\zeta^{i+1}_{m,n}=\omega^{i+1} e^{\iota( \varsigma_1m h_x+ \varsigma_2n h_y)},
\end{align}
where $\iota=\sqrt{-1},~\varsigma_1=\frac{2\pi k_1}{L},~\varsigma_2=\frac{2\pi k_2}{L}$.

It is readily apparent that
\begin{equation}\label{3.19}
	\left\{
	\begin{array}{lll}
\vspace{0.2cm}\mathcal H_x\zeta^{i+1}_{m,n}=\frac{1}{3}\left[\cos^2\left(\frac{\varsigma_1h_x}{2}\right)+2\right]\omega^{i+1} e^{\iota( \varsigma_1m h_x+ \varsigma_2n h_y)},\\\vspace{0.2cm}
\mathcal H_y\zeta^{i+1}_{m,n}=	\frac{1}{3}\left[\cos^2\left(\frac{\varsigma_2h_y}{2}\right)+2\right]\omega^{i+1} e^{\iota( \varsigma_1m h_x+ \varsigma_2n h_y)},\\\vspace{0.2cm}
\delta^2_x\zeta^{i+1}_{m,n}=-\frac{4}{h_x^2}\sin^2\left(\frac{\varsigma_1h_x}{2}\right)\omega^{i+1} e^{\iota( \varsigma_1m h_x+ \varsigma_2n h_y)},\\
\delta^2_y\zeta^{i+1}_{m,n}=-\frac{4}{h_y^2}\sin^2\left(\frac{\varsigma_2h_y}{2}\right)\omega^{i+1} e^{\iota( \varsigma_1m h_x+ \varsigma_2n h_y)}.
	\end{array}
\right.
\end{equation}
By inserting equation (\ref{3.18}) into equation (\ref{3.13}) and using equations given in (\ref{3.19}), we ascertain that
\begin{align}\label{3.20}
\nonumber &
\left(a_1+\frac{\sigma\lambda_1}{\mu}b_1\right)\left(a_2+\frac{\sigma\lambda_2}{\mu}b_2\right)\omega^{i+1}=\frac{a_1a_2}{\mu}\left[\sum_{k=1}^{i}\left(w^{(\alpha,\sigma)}_{i,k}-w^{(\alpha,\sigma)}_{i,k-1}\right)\omega^{k}+w^{(\alpha,\sigma)}_{i,0}\omega^{0}\right]\\&~~ -\frac{(1-\sigma)}{\mu}\left(\lambda_1a_2b_1+\lambda_2a_1b_2+\beta a_1a_2\right)\omega^{i}+\frac{\sigma^2}{\mu^2}\lambda_1\lambda_2b_1b_2\omega^{i},~0\leq i\leq N_t-1,
\end{align}
where 
\begin{align}\label{3.21}
	\left\{
	\begin{array}{lll}
	\vspace{0.2cm}	a_1=\frac{1}{3}\left[\cos^2\left(\frac{\varsigma_1h_x}{2}\right)+2\right],~
		b_1=\frac{4}{h_x^2}\sin^2\left(\frac{\varsigma_1h_x}{2}\right),\\\vspace{0.2cm}	a_2=\frac{1}{3}\left[\cos^2\left(\frac{\varsigma_2h_y}{2}\right)+2\right],~
		b_2=\frac{4}{h_y^2}\sin^2\left(\frac{\varsigma_2h_y}{2}\right).
	\end{array}
	\right.
\end{align}
Here, it is easy to understand that $a_1,a_2\geq \frac{2}{3}$, and $b_1,b_2\geq 0$.

Rewriting equation (\ref{3.20}), we have
\begin{align}\label{3.22}
	\nonumber	\omega^{i+1}=&\frac{1}{\left(a_1+\frac{\sigma\lambda_1}{\mu}b_1\right)\left(a_2+\frac{\sigma\lambda_2}{\mu}b_2\right)}\left[\frac{a_1a_2}{\mu}\left(\sum_{k=1}^{i}\left(w^{(\alpha,\sigma)}_{i,k}-w^{(\alpha,\sigma)}_{i,k-1}\right)\omega^{k}+w^{(\alpha,\sigma)}_{i,0}\omega^{0}\right)-\frac{(1-\sigma)}{\mu}\right.\\&~~~\times\left(\lambda_1a_2b_1+\lambda_2a_1b_2+\beta a_1a_2\right)\omega^{i}+\frac{\sigma^2}{\mu^2}\lambda_1\lambda_2b_1b_2\omega^{i}\Bigg],~0\leq i\leq N_t-1.
\end{align}

\begin{lemma}\label{L:6}
Consider $\omega^{i+1}$ to be the solution of equation (\ref{3.22}). Then, \begin{align*}|\omega^{i+1}|\leq |\omega^0|,~ 0\leq i\leq N_t-1.\end{align*}
\end{lemma}
\begin{proof}
We shall establish the validity of this statement by employing the principle of mathematical induction. By substituting $i=0$ into equation (\ref{3.22}), we obtain
\begin{align}\label{3.23}
	\omega^{1}=&\frac{\left[\frac{a_1a_2}{\mu}w^{(\alpha,\sigma)}_{0,0}-\frac{(1-\sigma)}{\mu}\left(\lambda_1a_2b_1+\lambda_2a_1b_2+\beta a_1a_2\right)+\frac{\sigma^2}{\mu^2}\lambda_1\lambda_2b_1b_2\right]\omega^{0}}{\left(a_1+\frac{\sigma\lambda_1}{\mu}b_1\right)\left(a_2+\frac{\sigma\lambda_2}{\mu}b_2\right)}.
\end{align}
Since $\sigma\geq(1-\sigma)\geq0$ and $w^{(\alpha,\sigma)}_{0,0}>0$, therefore
	\begin{align}\label{3.24}
		|\omega^{1}|\leq&\frac{\left[\frac{a_1a_2}{\mu}{(\beta\sigma+w^{(\alpha,\sigma)}_{0,0})}+\frac{\sigma}{\mu}\left(\lambda_1a_2b_1+\lambda_2a_1b_2\right)+\frac{\sigma^2}{\mu^2}\lambda_1\lambda_2b_1b_2\right]|\omega^{0}|}{\left(a_1+\frac{\sigma\lambda_1}{\mu}b_1\right)\left(a_2+\frac{\sigma\lambda_2}{\mu}b_2\right)}\leq|\omega^{0}|.
	\end{align}
Now, let us assume that
 	\begin{align}\label{3.25}
	|\omega^l|\leq |\omega^0|, ~\text{for all}~ 1\leq l\leq i.
	\end{align}
	Next, for $l=i+1$, from equation (\ref{3.22}) with Lemma \ref{L:1} and assumptions (\ref{3.25}), we get
	\begin{align*}
		\nonumber	|\omega^{i+1}|\leq&\frac{\left[\frac{a_1a_2}{\mu}\left(\displaystyle\sum_{k=1}^{i}\left(w^{(\alpha,\sigma)}_{i,k}-w^{(\alpha,\sigma)}_{i,k-1}\right)+w^{(\alpha,\sigma)}_{i,0}\right)+\frac{(1-\sigma)}{\mu}\left(\lambda_1a_2b_1+\lambda_2a_1b_2+\beta a_1a_2\right)+\frac{\sigma^2}{\mu^2}\lambda_1\lambda_2b_1b_2\right]}{\left(a_1+\frac{\sigma\lambda_1}{\mu}b_1\right)\left(a_2+\frac{\sigma\lambda_2}{\mu}b_2\right)}|\omega^0|.
	\end{align*}
	Since $\sigma\geq(1-\sigma)\geq0$ and $w^{(\alpha,\sigma)}_{i,i}>0$, therefore
	\begin{align*}
		\nonumber	|\omega^{i+1}|\leq&\frac{\left[\frac{a_1a_2}{\mu}\left(\beta\sigma+w^{(\alpha,\sigma)}_{i,i}\right)+\frac{\sigma}{\mu}\left(\lambda_1a_2b_1+\lambda_2a_1b_2\right)+\frac{\sigma^2}{\mu^2}\lambda_1\lambda_2b_1b_2\right]}{\left(a_1+\frac{\sigma\lambda_1}{\mu}b_1\right)\left(a_2+\frac{\sigma\lambda_2}{\mu}b_2\right)}|\omega^0|=|\omega^0|.
	\end{align*}
Thus, we have $|\omega^{i+1}|\leq|\omega^0|,$ for all $0\leq i\leq N_t-1$.
\end{proof}
\begin{Theorem}
The high-order compact ADI scheme given by (\ref{29}) for solving the model problem (\ref{7})-(\ref{8}) is unconditionally stable.
	\end{Theorem}
\begin{proof}
	By considering equation (\ref{3.17}) and Lemma \ref{L:6} together, one can obtain the following inequality
\begin{align*}
	\|\zeta^{i+1}\|^2_2=L^2\mathop{\sum^{\infty}\sum^{\infty}}_{k_1=-\infty\ k_2=-\infty}|\omega^{i+1}(k_1,k_2)|^2\leq L^2\mathop{\sum^{\infty}\sum^{\infty}}_{k_1=-\infty\ k_2=-\infty}|\omega^0(k_1,k_2)|^2=\|\zeta^{0}\|^2_2.
\end{align*}
Thus, $\|\zeta^{i+1}\|_2\leq \|\zeta^{0}\|_2,~ 0\leq i\leq N_t-1.$
Hence, the numerical scheme given by (\ref{29}) is unconditionally stable.
\end{proof}
\subsection{\bf{Convergence analysis}}
\noindent In this part, we aim to assess the convergence of the proposed scheme (\ref{29}) with respect to the model problem (\ref{7})-(\ref{8}).

By subtracting (\ref{29}) from (\ref{27}), we obtain
\begin{align}\label{3.26}
	\nonumber&\left(\mathcal H_x-\frac{\sigma\lambda_1}{\mu}\delta_{x}^{2}\right)\left(\mathcal H_y-\frac{\sigma\lambda_2}{\mu}\delta_{y}^{2}\right)\epsilon^{i+1}_{m,n}=\frac{1}{\mu}\left[\sum_{k=1}^{i}\left(w^{(\alpha,\sigma)}_{i,k}-w^{(\alpha,\sigma)}_{i,k-1}\right)\mathcal H_x\mathcal H_y\epsilon^{k}_{m,n}\right]-\frac{\beta(1-\sigma)}{\mu}  \\\nonumber&~~\times \mathcal H_x\mathcal H_y\epsilon^{i}_{m,n}+\frac{(1-\sigma)}{\mu}\left(\lambda_1 \mathcal H_y	\delta_{{x}}^{2}+\lambda_2  \mathcal H_x	\delta_{y}^{2}\right)\epsilon^{i}_{m,n}+\frac{\lambda_1\lambda_2\sigma^2}{\mu^2}\delta_{x}^{2}\delta_{y}^{2}\epsilon^{i}_{m,n}+ (\mathcal R_f)_{m,n}^{i+1},\\&~~~~~~~~~~~~~~~~~~~~~~~~~~~~~~~~ ~1\leq m\leq M_x-1,~1\leq n\leq M_y-1,~0\leq~ i\leq N_t-1,
\end{align}
where 
\begin{equation}\label{3.27}
	\left\{
	\begin{array}{lll}
		\vspace{0.2cm}\epsilon^{i}_{m,n}=v^{i}_{m,n}-V^{i}_{m,n},~~1\leq m\leq M_x-1,~1\leq n\leq M_y-1,~1\leq~ i\leq N_t,\\\vspace{0.2cm}
		 \epsilon^{0}_{m,n}=0,~0 \leq m \leq M_x, ~0 \leq n \leq M_y,\\\vspace{0.2cm}
		\epsilon^{i}_{0,n}=\epsilon^{i}_{M_x,n}, ~0 \leq n \leq M_y,~1\leq~ i\leq N_t,\\
	\epsilon^{i}_{m,0}=\epsilon^{i}_{m,M_y},~0 \leq m \leq M_x,~1\leq~ i\leq N_t.
	\end{array}
	\right.
\end{equation}
Now, we define the grid functions $\epsilon^{i+1}(x,y)$ and $\mathcal R_f^{i+1}(x,y)$ for $0\leq i\leq N_t-1$ as follows
\begin{align*}
	\epsilon^{i+1}(x,y)=
	\begin{cases}
		\epsilon^{i+1}_{m,n},& x_{m-\frac{1}{2}}<x\leq x_{m+\frac{1}{2}}, y_{n-\frac{1}{2}}<y\leq y_{n+\frac{1}{2}},\\
		& 1\leq m\leq M_x-1, 1\leq n\leq M_y-1,\\
		0,&0\leq x\leq \frac{h_x}{2}, L-\frac{h_x}{2}<x\leq L,\\
		&0\leq y\leq \frac{h_y}{2}, L-\frac{h_y}{2}<y\leq L,
	\end{cases}
\end{align*}
and
\begin{align*}
\mathcal R_f^{i+1}(x,y)=
	\begin{cases}
		(\mathcal R_f)^{i+1}_{m,n},& x_{m-\frac{1}{2}}<x\leq x_{m+\frac{1}{2}}, y_{n-\frac{1}{2}}<y\leq y_{n+\frac{1}{2}},\\
		& 1\leq m\leq M_x-1, 1\leq n\leq M_y-1 ,\\
		0, &0\leq x\leq \frac{h_x}{2}, L-\frac{h_x}{2}<x\leq L,\\
		& 0\leq y\leq \frac{h_y}{2}, L-\frac{h_y}{2}<y\leq L.
	\end{cases}
\end{align*}
The expressions of $\epsilon^{i+1}(x,y)$ and $\mathcal R_f^{i+1}(x,y)$ as a Fourier series are as follows 
\begin{align*}
	\epsilon^{i+1}(x,y)=\mathop{\sum^{\infty}\sum^{\infty}}_{k_1=-\infty\ k_2=-\infty}\xi^{i+1}(k_1,k_2)e^{2\pi\iota\left(\frac{k_1x+k_2y}{L}\right)},\\
	\mathcal R_f^{i+1}(x,y)=\mathop{\sum^{\infty}\sum^{\infty}}_{k_1=-\infty\ k_2=-\infty}\chi^{i+1}(k_1,k_2)e^{2\pi\iota\left(\frac{k_1x+k_2y}{L}\right)},
\end{align*}
where
 \begin{align*}
&	\xi^{i+1}(k_1,k_2)=\frac{1}{L^2}\int_0^L\int_0^L\epsilon^{i+1}(x,y)e^{-2\pi\iota\left(\frac{k_1x+k_2y}{L}\right)}dxdy,\\
&	\chi^{i+1}(k_1,k_2)=\frac{1}{L^2}\int_0^L\int_0^L \mathcal R
_f^{i+1}(x,y)e^{-2\pi\iota\left(\frac{k_1x+k_2y}{L}\right)}dxdy.
\end{align*}
By virtue of the $L_2$-discrete norm definition and Parseval's equality, we have
\begin{align}
&\label{3.28}	\|\epsilon^{i+1}\|^2_2=\mathop{\sum^{M_x-1}\sum^{M_y-1}}_{m=1\ n=1} h_x h_y|\epsilon^{i+1}_{m,n}|^2=L^2\mathop{\sum^{\infty}\sum^{\infty}}_{k_1=-\infty\ k_2=-\infty}|\xi^{i+1}(k_1,k_2)|^2,\\&
	\|\mathcal R_f^{i+1}\|^2_2=\mathop{\sum^{M_x-1}\sum^{M_y-1}}_{m=1\ n=1} h_x h_y|(\mathcal R_f)^{i+1}_{m,n}|^2=L^2\mathop{\sum^{\infty}\sum^{\infty}}_{k_1=-\infty\ k_2=-\infty}|\chi^{i+1}(k_1,k_2)|^2,\label{3.29}
\end{align}
for $0\leq i\leq N_t-1$, where
\begin{align*}&\epsilon^{i+1}= \left[\epsilon^{i+1}_{1,1},\epsilon^{i+1}_{1,2},\dots,\epsilon^{i+1}_{1,M_y-1},\epsilon^{i+1}_{2,1},\epsilon^{i+1}_{2,2},\dots,\epsilon^{i+1}_{2,M_y-1},\dots,\epsilon^{i+1}_{M_x-1,1},\epsilon^{i+1}_{M_x-1,2},\dots,\epsilon^{i+1}_{M_x-1,M_y-1}\right]^T,\\&\mathcal R_f^{i+1}= \left[(\mathcal R_f)^{i+1}_{1,1},(\mathcal R_f)^{i+1}_{1,2},\dots,(\mathcal R_f)^{i+1}_{1,M_y-1},(\mathcal R_f)^{i+1}_{2,1},(\mathcal R_f)^{i+1}_{2,2},\dots,(\mathcal R_f)^{i+1}_{2,M_y-1},\dots,(\mathcal R_f)^{i+1}_{M_x-1,1},\right.\\&~~~~~~~~~~~~~~~~~~~~~~~~~~~~~~~~~~~~~~~~~~~~~~~~~~~~~~~~~~~~~~~~~~~~~~~~~~(\mathcal R_f)^{i+1}_{M_x-1,2},\dots,(\mathcal R_f)^{i+1}_{M_x-1,M_y-1}\Big]^T.\end{align*}
Let the solutions $\epsilon^{i+1}_{m,n}$ and $(\mathcal R_f)^{i+1}_{m,n}$ have following form 
\begin{align}\label{3.30}
	\epsilon^{i+1}_{m,n}=\xi^{i+1} e^{\iota( \varsigma_1m h_x+\varsigma_2n h_y)}
	,~~
	(\mathcal R_f)^{i+1}_{m,n}=\chi^{i+1} e^{\iota( \varsigma_1m h_x+\varsigma_2n h_y)},
\end{align}
where $\iota=\sqrt{-1},~\varsigma_1=\frac{2\pi k_1}{L},~\varsigma_2=\frac{2\pi k_2}{L}$.

It is readily apparent that
\begin{equation}\label{3.31}
	\left\{
	\begin{array}{lll}
		\vspace{0.2cm}\mathcal H_x\epsilon^{i}_{m,n}=\frac{1}{3}\left[\cos^2\left(\frac{\varsigma_1h_x}{2}\right)+2\right]\xi^i e^{\iota( \varsigma_1m h_x+\varsigma_2n h_y)},\\\vspace{0.2cm}
		\mathcal H_y\epsilon^{i}_{m,n}=	\frac{1}{3}\left[\cos^2\left(\frac{\varsigma_2 h_y}{2}\right)+2\right]\xi^i e^{\iota( \varsigma_1m h_x+\varsigma_2n h_y)},\\\vspace{0.2cm}
		\delta^2_x\epsilon^{i}_{m,n}=-\frac{4}{h_x^2}\sin^2\left(\frac{\varsigma_1h_x}{2}\right)\xi^i e^{\iota( \varsigma_1m h_x+\varsigma_2n h_y)},\\
		\delta^2_y\epsilon^{i}_{m,n}=-\frac{4}{h_y^2}\sin^2\left(\frac{\varsigma_2h_y}{2}\right)\xi^i e^{\iota( \varsigma_1m h_x+\varsigma_2n h_y)}.
	\end{array}
	\right.
\end{equation}
By inserting equation (\ref{3.30}) into equation (\ref{3.26}), using $\xi^0=0$, and equations given in (\ref{3.31}), we ascertain that
\begin{align}\label{3.32}
	 &
	\xi^{i+1}=\frac{\frac{a_1a_2}{\mu}\left(\displaystyle\sum_{k=1}^{i}(w^{(\alpha,\sigma)}_{i,k}-w^{(\alpha,\sigma)}_{i,k-1})\xi^{k}\right) -\frac{(1-\sigma)}{\mu}\left(\lambda_1a_2b_1+\lambda_2a_1b_2+\beta a_1a_2\right)\xi^{i}+\frac{\sigma^2}{\mu^2}\lambda_1\lambda_2b_1b_2\xi^{i}+\chi^{i+1}}{\left(a_1+\frac{\sigma\lambda_1}{\mu}b_1\right)\left(a_2+\frac{\sigma\lambda_2}{\mu}b_2\right)}.
\end{align}
As a result of the convergence exhibited by the series on the right hand side of equation (\ref{3.29}), it follows that there exists a positive value $d>0$ satisfying
\begin{align}\label{3.33}
	|\chi^{i+1}|\equiv |\chi^{i+1}(k_1,k_2)|\leq d\tau|\chi^1(k_1,k_2)|\equiv d\tau|\chi^1|,~0\leq i\leq N_t-1,
\end{align}
where $\tau = \displaystyle\max_{1\leq i\leq N_t}\tau_i.$
\begin{lemma}\label{L:7} 
	For $i=0,1,\dots,N_t-1$, we have
	\begin{align}\label{3.34} 
		|\xi^{i+1}|\leq \frac{9}{4}d(1+\tau)^{i+1}|\chi^1|,~0\leq i\leq N_t-1,
	\end{align}
	where $d$ is a constant in (\ref{3.33}).
\end{lemma}
\begin{proof}
We shall establish the validity of this statement by employing the principle of mathematical induction. By substituting $i=0$ and using $\xi^0=0$ into equation (\ref{3.32}), we obtain 
	\begin{align*}
		&
		\xi^{1}=\frac{\chi^{1}}{\left(a_1+\frac{\sigma\lambda_1}{\mu}b_1\right)\left(a_2+\frac{\sigma\lambda_2}{\mu}b_2\right)}.
	\end{align*}
	By utilizing equation (\ref{3.34}) along with the provided conditions $a_1,a_2 \geq \frac{2}{3}$ and $b_1,b_2 \geq 0$, we can conclude  
	\begin{align}\label{3.35}
		|\xi^1|\leq \frac{9}{4}d(1+\tau)|\chi^1|.
	\end{align}
	Now, let us assume that \begin{align}\label{3.36}
		|\xi^l|\leq \frac{9}{4}d(1+\tau)^l|\chi^1|, ~\text{for all}~ 1\leq l\leq i.
	\end{align} 
	Next, for $l=i+1$, from equation (\ref{3.32}) with Lemma \ref{L:1}, and equation (\ref{3.33}), we get
	\begin{align}
	\nonumber	|\xi^{i+1}|=&\frac{1}{\left(a_1+\frac{\sigma\lambda_1}{\mu}b_1\right)\left(a_2+\frac{\sigma\lambda_2}{\mu}b_2\right)}\left[\frac{a_1a_2}{\mu}\left(\displaystyle\sum_{k=1}^{i}\left(w^{(\alpha,\sigma)}_{i,k}-w^{(\alpha,\sigma)}_{i,k-1}\right)\right)\displaystyle \max_{1\leq j\leq i}|\xi^{j}| +\frac{(1-\sigma)}{\mu}\right.\\	\nonumber&~\times\left(\lambda_1a_2b_1+\lambda_2a_1b_2+\beta a_1a_2\right)|\xi^{i}|+\frac{\sigma^2}{\mu^2}\lambda_1\lambda_2b_1b_2|\xi^{i}|+d\tau|\chi^{1}|\Bigg].\end{align}
	Further, using the assumptions (\ref{3.36}) and $\sigma\geq(1-\sigma)\geq0$, we get
		\begin{align*}
	|\xi^{i+1}|&\leq	\frac{1}{\left(a_1+\frac{\sigma\lambda_1}{\mu}b_1\right)\left(a_2+\frac{\sigma\lambda_2}{\mu}b_2\right)}\left[\left(\frac{a_1a_2}{\mu}\left(w^{(\alpha,\sigma)}_{i,i}-w^{(\alpha,\sigma)}_{i,0}\right) +\frac{\sigma}{\mu}\left(\lambda_1a_2b_1+\lambda_2a_1b_2+\beta a_1a_2\right)\right.\right.\\&~+\frac{\sigma^2}{\mu^2}\lambda_1\lambda_2b_1b_2\bigg)\left( \frac{9}{4}d(1+\tau)^i|\chi^1|\right)+\frac{9}{4}d\tau|\chi^{1}|\Bigg].
	\end{align*}
	Again, by invoking Lemma \ref{L:1}, we have
		\begin{align*}
		|\xi^{i+1}|&\leq  \frac{9}{4}d\Bigl\{(1+\tau)^i+\tau\Bigr\}|\chi^{1}|\leq \frac{9}{4}d(1+\tau)^{i+1}|\chi^1|.
	\end{align*} 
	Thus, we have the Lemma.
\end{proof}
\begin{Theorem}
Assume that the problem (\ref{7})-(\ref{8}) has a solution $v(x, y, t)$, which meets the assumptions provided in (\ref{4}) and (\ref{5}) and let $V=\{V^i_{m,n}~|~ 0\leq m\leq M_x, 0\leq n\leq  M_y, 
1\leq i\leq  N_t$\} be the solution of the high-order compact ADI scheme given by (\ref{29}). Then, we have
\begin{align*} 
	\|v^{i+1}-V^{i+1}\|_2\leq&~\hat c
	\left(N_t^{-\min\{3-\alpha,\theta\alpha,1+2\alpha,2+\alpha\}}+h_x^4+h_y^4\right),~0\leq i\leq N_t-1,
\end{align*}
where $\hat c=
\frac{9}{4}\tilde c	d\exp({2T_f})$.
	\end{Theorem}
	\begin{proof}
Utilizing both Theorem \ref{T:3.1} and the first equality of (\ref{3.29}), we arrive at
\begin{align}\label{3.37}
\nonumber	\|\mathcal R^{i+1}_f\|_2\leq&~c \sqrt{M_x h_x}\sqrt{M_y h_y}\left(N_t^{-\min\{3-\alpha,\theta\alpha,1+2\alpha,2+\alpha\}}+h_x^4+h_y^4\right)\\\leq&~\tilde c\left(N_t^{-\min\{3-\alpha,\theta\alpha,1+2\alpha,2+\alpha\}}+h_x^4+h_y^4\right),~0\leq i\leq N_t-1,
\end{align}
where $\tilde c=cL$.
By invoking Lemma \ref{L:7}, along with equations (\ref{3.28}), (\ref{3.29}), and (\ref{3.37}), for $0\leq i\leq N_t-1 $, we can deduce the following estimate
\begin{align*}
	\|\epsilon^{i+1}\|^2_2=&~L^2\mathop{\sum^{\infty}\sum^{\infty}}_{k_1=-\infty\ k_2=-\infty}|\xi^{i+1}(k_1,k_2)|^2\\
	\leq&~L^2\mathop{\sum^{\infty}\sum^{\infty}}_{k_1=-\infty\ k_2=-\infty}\left(\frac{9}{4}d\right)^2(1+\tau)^{2(i+1)}|\chi^1(k_1,k_2)|^2\\=&~
	\left(\frac{9}{4}d\right)^2(1+\tau)^{2(i+1)}\|\mathcal R_f^{1}\|^2_2\\\leq&~\tilde c^2	
	\left(\frac{9}{4}d\right)^2(1+\tau)^{2(i+1)}\left(N_t^{-\min\{3-\alpha,\theta\alpha,1+2\alpha,2+\alpha\}}+h_x^4+h_y^4\right)^2.
\end{align*}
Moreover, by noting the fact that $\tau,i\tau\leq T_f,$ we obtain
\begin{align*} 
		\|\epsilon^{i+1}\|_2\leq&~\hat c
		\left(N_t^{-\min\{3-\alpha,\theta\alpha,1+2\alpha,2+\alpha\}}+h_x^4+h_y^4\right),
\end{align*}
where $\hat c=
\frac{9}{4}\tilde c	d\exp({2T_f})$.

Henceforth, the theorem holds.
\end{proof}

\section {Conclusion}\label{con}
\noindent The focus of this study was to provide a highly efficient high-order numerical method for numerically solving the two-dimensional time-fractional convection-diffusion equation. A notable aspect of this problem was the utilization of the Caputo time-fractional derivative, which led to an initial weak singularity at $t=0$. The core foundation of the proposed scheme centered around the utilization of two key components: the adoption of the L2-1$_\sigma$ approximation formula for the Caputo time-fractional derivative implemented on a fitted mesh and a high-order compact finite difference approximation for the space derivatives on a uniform mesh. Moreover, to solve the resulting system, a two-step ADI (Alternating Direction Implicit) approach is employed. In addition, a rigorous theoretical analysis has been undertaken, thoroughly assessing the stability and convergence of the proposed scheme. Our findings demonstrated that the method exhibits unconditional stability for every $\alpha\in(0,1)$, accompanied by uniform convergence with an order of $\mathcal O\left(N_t^{-\min\{3-\alpha,\theta\alpha,1+2\alpha,2+\alpha\}}+h_x^4+h_y^4\right)$. 
\section*{Declaration of competing interest}
 Conflict of interest: The authors have no relevant financial or non-financial interests to disclose.
\section*{Data availability}
No data was used for the research described in the article.

\newpage
	\bibliographystyle{elsarticle-num}
	\bibliography{references}

\end{document}